\documentclass[reqno]{amsart}
\usepackage{amssymb,eucal,latexsym,mathrsfs,tikz-cd,enumerate,graphicx}

\newenvironment{enumeratei}{\begin{enumerate}[\upshape (i)]}%
{\end{enumerate}}
{\end{enumerate}}
\newenvironment{enumerater}{\begin{enumerate}[\upshape (1)]}%
{\end{enumerate}}

\newcommand{\oo}[1]{\left]{#1}\right[}

\newcommand{\sdf}{\mathbin{\text{\rotatebox[origin=c]{90}{$\oslash$}}}}
\newcommand{\todot}{\overset{\boldsymbol{.}}{\rightarrow}}
\newcommand{\pup}[1]{\textup{(}{#1}\textup{)}}
\newcommand{\lgrp}{$\ell$-group}
\newcommand{\lidl}{$\ell$-i\-de\-al}
\newcommand{\jirr}{join-ir\-re\-duc\-i\-ble}
\newcommand{\mirr}{meet-ir\-re\-duc\-i\-ble}

\newcommand{\jh}{join-ho\-mo\-mor\-phism}
\newcommand{\mh}{meet-ho\-mo\-mor\-phism}
\newcommand{\eqdef}{\overset{\mathrm{def}}{=}}

\DeclareMathOperator{\Spec}{Spec}

\newcommand{\FL}{\operatorname{F}_{\ell}}

\DeclareMathOperator{\Down}{Down}

\DeclareMathOperator{\BR}{BR}
\DeclareMathOperator{\Ji}{Ji}
\DeclareMathOperator{\Mi}{Mi}
\DeclareMathOperator{\card}{card}

\newcommand{\ga}{\alpha}
\newcommand{\gb}{\beta}

\newcommand{\gd}{\delta}
\newcommand{\gf}{\varphi}

\newcommand{\gl}{\lambda}

\newcommand{\go}{\omega}
\newcommand{\eps}{\varepsilon}

\newcommand{\bck}[1]{[\![{#1}]\!]}

\newcommand{\gO}{\Omega}

\newcommand{\sd}{\mathbin{\smallsetminus}}

\newcommand{\jz}{$(\vee,0)$}

\newcommand{\jzh}{\jz-ho\-mo\-mor\-phism}

\newcommand{\pI}[1]{\bigl({#1}\bigr)}
\newcommand{\pII}[1]{\Bigl({#1}\Bigr)}

\newcommand{\set}[1]{\left\{#1\right\}}
\newcommand{\setm}[2]{\set{{#1}\mid{#2}}}
\newcommand{\vecm}[2]{\left({#1}\mid{#2}\right)}
\newcommand{\seq}[1]{\langle{#1}\rangle}

\newcommand{\dnw}{\mathbin{\downarrow}}

\newcommand{\es}{\varnothing}
\newcommand{\res}{\mathbin{\restriction}}

\newcommand{\EE}{\mathbb{E}}

\newcommand{\NN}{\mathbb{N}}

\newcommand{\ip}[1]{{#1}^{\sqcup\infty}}

\DeclareMathOperator{\Bool}{Bool}

\DeclareMathOperator{\Op}{Op}
\newcommand{\Ops}{\Op^-}

\DeclareMathOperator{\Idc}{Id_c}

\DeclareMathOperator{\rng}{rng}
\DeclareMathOperator{\Pow}{Pow}

\newcommand{\cD}{{\mathcal{D}}}
\newcommand{\cE}{{\mathcal{E}}}
\newcommand{\cF}{{\mathcal{F}}}
\newcommand{\cG}{{\mathcal{G}}}

\newcommand{\cL}{{\mathcal{L}}}

\numberwithin{equation}{section}

\newtheorem*{stat}{\name}
\newcommand{\name}{testing}

\theoremstyle{plain}

\newtheorem{theorem}{Theorem}[section]
\newtheorem{proposition}[theorem]{Proposition}
\newtheorem{corollary}[theorem]{Corollary}
\newtheorem{lemma}[theorem]{Lemma}

\theoremstyle{definition}

\newtheorem{definition}[theorem]{Definition}
\newtheorem{notation}[theorem]{Notation}

\newtheorem*{problem}{Problem}

\theoremstyle{remark}

\newtheorem*{note}{Note}

\newcommand{\qedc}{{\qed}~{\rm Claim~{\theclaim}.}}
\newcommand{\qedsc}{{\qed}~{\rm Claim.}}

\numberwithin{figure}{section}
\numberwithin{table}{section}

\newcommand{\kk}{\Bbbk}
\newcommand{\kkp}[1]{\kk^{(#1)}}

\newcommand{\scL}{\mathbin{\mathscr{L}}}

\title[Spectral subspaces]{Spectral subspaces of spectra of Abelian lattice-ordered groups in size aleph one}

\author[M. Plo\v{s}\v{c}ica]{Miroslav Plo\v{s}\v{c}ica}
\address{Faculty of Natural Sciences\\
\v{S}af\'arik's University\\
Jesenn\'a 5\\
04154 Ko\v{s}ice\\
Slovakia}
\email{miroslav.ploscica@upjs.sk}
\urladdr{https://ploscica.science.upjs.sk}

\author[F. Wehrung]{Friedrich Wehrung}
\address{Normandie Universit\'e, UNICAEN\\
CNRS UMR 6139, LMNO\\
14000 Caen\\
France}
\email{friedrich.wehrung01@unicaen.fr}
\urladdr{https://wehrungf.users.lmno.cnrs.fr}

\date{\today}



\begin{document}

\begin{abstract}
It is well known that the lattice~$\Idc{G}$ of all principal \lidl{s} of any Abelian \lgrp~$G$ is a completely normal distributive $0$-lattice, and that not every completely normal distributive $0$-lattice is a homomorphic image of some~$\Idc{G}$, \emph{via} a counterexample of cardinality~$\aleph_2$.
We prove that every completely normal distributive $0$-lattice with at most~$\aleph_1$ elements is a homomorphic image of some~$\Idc{G}$.
By Stone duality, this means that every completely normal generalized spectral space, with at most~$\aleph_1$ compact open sets, is homeomorphic to a spectral subspace of the $\ell$-spectrum of some Abelian \lgrp.
\end{abstract}

\maketitle


\section{Introduction}\label{S:Intro}

A subset~$I$, in a lattice-ordered group (in short \emph{\lgrp})~$G$, is an \emph{\lidl} if it is an order-convex normal subgroup closed under the lattice operations.
If~$I\neq G$, we say in addition that~$I$ is \emph{prime} if $x\wedge y\in I$ implies that $\set{x,y}\cap I\neq\es$, whenever $x,y\in G$.
In case~$G$ is \emph{Abelian}, the \emph{$\ell$-spectrum} of~$G$ is defined as the set~$\Spec{G}$ of all prime \lidl{s} of~$G$, endowed with the topology whose closed subsets are the $\setm{P\in\Spec{G}}{X\subseteq P}$ for $X\subseteq G$ (often called the \emph{hull-kernel topology}).
Denote by~$\cG$ the class of all Abelian \lgrp{s}.

The problem of the description of $\ell$-spectra of all Abelian \lgrp{s} (say the \emph{$\ell$-spectrum problem}) is stated, in the language of MV-algebras, in Mundici \cite[Problem~2]{Mund2011}.
Now under Stone duality (cf. Gr\"atzer~\cite[\S~II.5]{LTF}, Johnstone \cite[\S~II.3]{Johnst1982}, Rump and Yang~\cite{RumYan2009} for the case without top element, and Wehrung \cite[\S~2.2]{RAlg} for a summary), for any $G\in\cG$, $\Spec{G}$ corresponds to the lattice~$\Idc{G}$ of all principal \lidl{s} of~$G$; that is, $\Idc{G}=\setm{\seq{a}}{a\in G^+}$ where each $\seq{a}\eqdef\setm{x\in G}{(\exists n\in\NN)(|x|\leq na)}$.
This enables us to restate the $\ell$-spectrum problem as the description problem of the class $\Idc\cG\eqdef\setm{D}{(\exists G\in\cG)(D\cong\Idc{G})}$.
All such lattices are clearly distributive with smallest element (usually denoted by~$0$).
They are also \emph{completely normal} (cf. Bigard, Keimel, and Wolfenstein \cite[Ch.~10]{BKW}), that is, they satisfy the statement
 \[
 (\forall a,b)(\exists x,y)(a\vee b=a\vee y=x\vee b\text{ and }
 x\wedge y=0)\,.
 \]
Delzell and Madden observed in~\cite[Theorem~2]{DelMad1994}, \emph{via} a counterexample of cardinality~$\aleph_1$, that those properties are not sufficient to characterize~$\Idc\cG$.
On the other hand, the second author proved in~\cite{MV1} that every \emph{countable} completely normal distributive $0$-lattice belongs to~$\Idc\cG$.
The categorical concept of \emph{condensate}, initiated in the second author's work~\cite{Larder} with Pierre Gillibert, together with the main result of~\cite{Ceva}, enabled the second author to prove in~\cite{NonElt} that~$\Idc\cG$ is not the class of models of any class of~$\scL_{\infty\gl}$ sentences of lattice theory, for any infinite cardinal~$\gl$.
Using further tools from infinitary logic, the second author extended those results in~\cite{AccProj} by proving that~$\Idc\cG$ is not the \emph{complement of a projective class} over~$\scL_{\infty\infty}$, thus verifying in particular that the additional property of all lattices~$\Idc{G}$ coined by the first author in his proof of~\cite[Theorem~2.1]{Plo21} is still not sufficient to characterize~$\Idc\cG$.

As observed in the above-cited references, all those results extend to the class of all (lattice) \emph{homomorphic images} of lattices~$\Idc{G}$.
On the other hand, not every homomorphic image of a lattice of the form~$\Idc{G}$ belongs to~$\Idc\cG$ (cf. Wehrung \cite[Example~10.6]{MV1}).
Recast in terms of spectra, \emph{via} Stone duality, this means that \emph{not every spectral subspace of an $\ell$-spectrum is an $\ell$-spectrum}.

Moreover, not every completely normal bounded distributive lattice is a homomorphic image of some~$\Idc{G}$: a counterexample of cardinality~$\aleph_2$ is constructed in Wehrung~\cite{Ceva}.

In this paper we complete the picture above, by establishing that \emph{every completely normal distributive $0$-lattice~$D$, with at most~$\aleph_1$ elements, is a homomorphic image of~$\Idc{G}$ for some Abelian \lgrp~$G$}.
This also strengthens the first author's result, obtained in~\cite{Plo21}, that~$D$ is \emph{Cevian}.
In fact, we verify the slightly more general statement that~$G$ may be taken a vector lattice over any given countable totally ordered division ring~$\kk$ (cf. Theorem~\ref{T:Al1hom}), modulo the obvious change in the definition of an \lidl\ (i.e., \lidl{s} need to be closed under scalar multiplication by elements of~$\kk$; see Wehrung \cite[\S~2.3]{RAlg} for more detail).
Due to the results of \cite[\S~9]{RAlg}, the countability assumption on~$\kk$ cannot be dispensed with.

Our argument will roughly follow the one from Wehrung~\cite{MV1}, with the ``Main Extension Lemma'' \cite[Lemma~4.2]{MV1} strengthened from finite lattices to certain infinite lattices, and streamlined \emph{via} the introduction of \emph{consonance kernels} (cf. Definition~\ref{D:ConsKer}), as Lemma~\ref{L:MainExt}.
The proof of the ``closure step'' \cite[Lemma~7.2]{MV1} fails in that more general context, so we get only ``homomorphic image'' as opposed to ``isomorphic copy'', of~$\Idc{G}$.
This will also require a few known additional properties of finite distributive lattices and their homomorphisms, \emph{via} Birkhoff duality (see in particular Lemma~\ref{L:ClosedJL}).
Our final argument, given a completely normal distributive $0$-lattice~$L$, will start by expressing~$L$ as a directed union of an ascending $\go_1$-sequence $\vec{L}=\vecm{L_{\xi}}{\xi<\go_1}$ of countable completely normal distributive $0$-lattices, and then, with the help of Lemma~\ref{L:MainExt}, iteratively lift all subdiagrams $\vecm{L_{\xi}}{\xi<\ga}$, with $\ga<\go_1$, with respect to the functor~$\Idc$.
That part of our argument turns out to be valid not only for the chain~$\go_1$ but for any tree in which every element has countable height (cf. Theorem~\ref{T:Part3Ext}).

\section{Basic concepts}\label{S:Basic}

\subsection{Sets, posets, lattices}
For any set~$X$, $\Pow{X}$ denotes the powerset algebra of~$X$.
By ``countable'' we will mean ``at most countable''.
For an element~$a$ in a partially ordered set (from now on~\emph{poset})~$P$, we set $P\dnw a\eqdef\setm{p\in P}{p\leq a}$ (or $\dnw a$ if~$P$ is understood).
A subset~$A$ of~$P$ is a \emph{lower subset} of~$P$ if $P\dnw a\in A$ whenever $a\in A$.
A poset~$P$ with bottom element is a \emph{tree} if~$P\dnw a$ is well-ordered under the induced order whenever $a\in P$.

For a subset~$P$ in a poset~$Q$ and for $x\in Q$, $x^P$ (resp., $x_P$) denotes the least $y\in P$ such that $x\leq y$ (resp., the largest $y\in P$ such that $y\leq x$) if it exists.
We say that~$P$ is \emph{relatively complete in~$Q$} if~$x^P$ and~$x_P$ both exist for all~$x\in P$.
If~$P$ is a subalgebra of a Boolean algebra~$Q$, it suffices to verify that~$x^P$ exists whenever $x\in Q$ (resp., $x_P$ exists whenever $x\in Q$).

Relative completeness has been used in a description of projective Boolean algebras.  For the proof of the following
(easy) assertion see Heindorf and Shapiro~\cite[Lemma 1.2.7]{HeiSha1994}.

\begin{lemma}\label{L:RelCplBA}
Let~$A$, $A'$ be subalgebras of a Boolean algebra~$B$ with~$A'$ finitely generated over~$A$.
If~$A$ is relatively complete in~$B$, then so is~$A'$.
\end{lemma}

For posets~$P$ and~$Q$ with respective top elements~$\top_P$ and~$\top_Q$, a map $f\colon P\to Q$ is \emph{top-faithful} if $f^{-1}\set{\top_Q}=\set{\top_P}$.
For any poset~$P$, $\ip{P}$ denotes the poset obtained by adding an extra element, usually denoted by~$\infty$, atop of~$P$.
For any map $f\colon P\to Q$, we denote by $\ip{f}\colon\ip{P}\to\ip{Q}$ the unique extension of~$f$ sending~$\infty$ to~$\infty$.
Such maps are exactly the top-faithful maps from~$\ip{P}$ to~$\ip{Q}$.

We denote by~$\Ji{L}$ (resp., $\Mi{L}$) the set of all \jirr\ (resp., \mirr) elements in a lattice~$L$, endowed with the induced ordering.
For any \jirr\ element~$p$ in a finite distributive lattice~$D$, we denote by~$p_*$ the unique lower cover of~$p$ in~$D$, and by~$p^{\dagger}$ the largest element of~$D$ not above~$p$; so $p_*=p\wedge p^{\dagger}$.
The assignment $p\mapsto p^{\dagger}$ defines an order-isomorphism from~$\Ji{D}$ onto~$\Mi{D}$.

As in Wehrung~\cite{MV1,RAlg}, two elements~$a$ and~$b$ in a $0$-lattice (i.e., lattice with a bottom element)~$D$ are \emph{consonant} if there exist $u,v\in D$ such that $a\leq u\vee b$, $b\leq a\vee v$, and $u\wedge v=0$.
A subset~$X$ of~$D$ is consonant if any pair of elements in~$X$ is consonant.
The lattice~$D$ is \emph{completely normal} if it is consonant within itself.

We denote by~$\Ji{L}$ (resp., $\Mi{L}$) the set of all \jirr\ (resp., \mirr) elements in a lattice~$L$, endowed with the induced partial ordering.
The assignment $D\mapsto\Ji{D}$ is part of \emph{Birkhoff duality} between finite distributive lattices, with $0,1$-lattice homomorphisms, and finite posets, with isotone maps (cf. Gr\"atzer \cite[\S~II.1.3]{LTF}).
The Birkhoff dual of a $0,1$-lattice homomorphism $\gf\colon D\to E$ is the map $\Ji{E}\to\Ji{D}$, $q\mapsto q^{\gf}\eqdef\min\setm{x\in D}{q\leq\gf(x)}$.

For any distributive $0$-lattice~$D$, we denote by~$\BR(D)$ the generalized Boolean algebra \emph{R-generated by~$D$} in the sense of Gr\"atzer \cite[\S~II.4]{LTF} (aka the \emph{Boolean envelope} of~$D$).
Equivalently, $\BR(D)$ is the universal generalized Boolean algebra of~$D$.
Up to isomorphism, $\BR(D)$ is the unique generalized Boolean algebra generated by~$D$ as a $0$-sublattice.
The assignment $D\mapsto\BR(D)$ canonically extends to a functor, which turns $0$-lattice embeddings to embeddings of generalized Boolean algebras.
For a $0$-sublattice~$D$ of a distributive lattice~$E$ with~$0$, we will thus identify~$\BR(D)$ with its canonical image in~$\BR(E)$.
If~$D$ is a finite distributive lattice and $P\eqdef\Ji{D}$, then the assingment $x\mapsto P\dnw{x}$ defines an isomorphism from~$D$ onto the lattice $\Down{P}$ of all lower subsets of~$P$.
Since the universal Boolean algebra of~$\Down{P}$ is the powerset lattice of~$P$, with each $\set{p}=(\dnw{p})\setminus(\dnw{p})_*$, it follows that the atoms of~$\BR(D)$ are exactly the $p\wedge\neg p_*$ for $p\in\Ji{D}$.

\begin{lemma}\label{L:BasicBR}
The following statements hold, for any distributive $0$-lattice~$D$:
\begin{enumerater}
\item\label{a1b1a2b2}
For all $a_1,a_2,b_1,b_2\in D$, $a_1\wedge\neg b_1\leq a_2\wedge\neg b_2$ within~$\BR(D)$ if{f} $a_1\leq a_2\vee b_1$ and $a_1\wedge b_2\leq b_1$ within~$D$.

\item\label{abpp_*}
If~$D$ is finite, then $a\wedge\neg b=\bigvee\setm{p\wedge\neg p_*}{p\in\Ji{D}\,,\ p\leq a\,,\ p\nleq b}$ within~$\BR(D)$, whenever $a,b\in D$.
\end{enumerater}
\end{lemma}

\begin{lemma}\label{L:BR2Ineq}
Let~$D$ and~$L$ be distributive $0$-lattices with~$D$ finite, let $\gf\colon D\to L$ be a $0$-lattice homomorphism, let $a,b\in D$, and let $c\in L$.
Then $\gf(a)\leq\gf(b)\vee c$ if{f} $\gf(p)\leq\gf(p_*)\vee c$ whenever $p\in\Ji{D}$ with $p\leq a$ and $p\nleq b$.
\end{lemma}

\begin{proof}
$\gf(a)\leq\gf(b)\vee c$ if{f} $\BR(\gf)(a\wedge\neg b)\leq c$, if{f} $\BR(\gf)(p\wedge\neg p_*)\leq c$ whenever $p\in\Ji{D}$ such that~$p\leq a$ and~$p\nleq b$ (we apply Lemma~\ref{L:BasicBR}\eqref{abpp_*}).
Now $\BR(\gf)(p\wedge\neg p_*)\leq c$ if{f} $\gf(p)\leq\gf(p_*)\vee c$.
\end{proof}

For any elements~$x$ and~$y$ in a lattice~$E$ let $x\rightarrow_Ey$ denote the largest $z\in E$, if it exists, such that $x\wedge z\leq y$ (it is also called the \emph{pseudocomplement of $x$ relative to $y$}); so~$\rightarrow_E$ is the \emph{Heyting implication} on~$E$.
If~$\rightarrow_E$ is defined on every pair of elements then we say that~$E$ is a \emph{generalized Heyting algebra}.
If, in addition, $E$ has a bottom element, then we say that~$E$ is a \emph{Heyting algebra}.
Every Heyting algebra is a bounded distributive lattice, and every finite distributive lattice is a Heyting algebra%
\footnote{
Strictly speaking we should set the Heyting implication~$\rightarrow$ apart from the lattice signature, thus for example stating that ``every finite distributive lattice expands to a unique Heyting algebra''.
The shorter formulation, which we shall keep for the sake of simplicity, reflects a standard abuse of terminology that should not create any confusion here.
}
.

Dually, we denote by $x\sd_Ey$ the least $z\in E$, if it exists, such that $x\leq y\vee z$.
It is the \emph{dual pseudocomplement of $x$ relative to $y$}.

A lattice homomorphism $\gf\colon D\to E$ is \emph{closed} if whenever $a_0,a_1\in D$ and $b\in E$, if $\gf(a_0)\leq\gf(a_1)\vee b$, then there exists $x\in D$ such that $a_0\leq a_1\vee x$ and $\gf(x)\leq b$.
If~$\gf$ is an inclusion map we will say that~$D$ is a \emph{closed sublattice} of~$E$.

The following folklore lemma, whose easy proof we leave to the reader as an exercise, enables to read, on the Birkhoff dual, whether a given homomorphism, between finite distributive lattices, is a homomorphism of Heyting algebras or a closed homomorphism, respectively.

\begin{lemma}\label{L:ClosedJL}
The following statements hold, for any finite distributive lattices~$D$ and~$E$ and any $0,1$-lattice homomorphism $\gf\colon D\to E$:
\begin{enumerater}
\item\label{CharHey}
$\gf$ is a homomorphism of Heyting algebras if{f} for all $p\in\Ji{D}$ and all $q\in\Ji{E}$, if $p\leq q^{\gf}$, then there exists $x\in\Ji{E}$ such that $x\leq q$ and $x^{\gf}=p$.

\item\label{CharClosed}
$\gf$ is closed if{f} for all $p\in\Ji{D}$ and all $q\in\Ji{E}$, if $q^{\gf}\leq p$, then there exists $x\in\Ji{E}$ such that $q\leq x$ and $x^{\gf}=p$.
\end{enumerater}
\end{lemma}

\subsection{The lattices $\Bool(\cF,\gO)$, $\Op(\cF,\gO)$, and $\Ops(\cF,\gO)$}
For more detail on this subsection we refer the reader to Wehrung \cite{MV1,RAlg}.
For a right vector space~$\EE$ over a totally ordered division ring~$\kk$, a map $f\colon\EE\to\nobreak\kk$ is an \emph{affine functional} if $f-f(0)$ is a linear functional.
Note that the affine functionals on~$\EE$ form a \emph{left} vector space over~$\kk$.

For functions~$f$ and~$g$ with common domain~$\gO$ and values in a poset~$T$, we set $\bck{f\leq g}\eqdef\setm{x\in\gO}{f(x)\leq g(x)}$; and similarly for $\bck{f<g}$, $\bck{f=g}$, $\bck{f\neq g}$, and so on.
Throughout this paper, $f$ and~$g$ will always be restrictions, to a convex set~$\gO$, of continuous affine functionals on a topological vector space~$\EE$ over a totally ordered division ring~$\kk$.
For a set~$\cF$ of maps from~$\gO$ to~$\kk$, we will denote by~$\Bool(\cF,\gO)$ the Boolean subalgebra of the powerset of~$\gO$ generated by all subsets~$\bck{f>0}$ and~$\bck{f<0}$ for $f\in\cF$.
As in~\cite{RAlg}, we will also denote by~$\Ops(\cF,\gO)$ the sublattice of~$\Bool(\cF,\gO)$ generated by all $\bck{f>0}$ and~$\bck{f<0}$ where $f\in\cF$, and then set $\Op(\cF,\gO)\eqdef\Ops(\cF,\gO)\cup\nobreak\set{\gO}$.
Evidently, $\Bool(\cF,\gO)$ is generated, as a Boolean algebra, by its $0$-sublattice $\Op(\cF,\gO)$; so $\Bool(\cF,\gO)=\BR\pI{\Op(\cF,\gO)}$.

For any set~$I$ and any totally ordered division ring~$\kk$, we will occasionally identify every element $a=\vecm{a_i}{i\in I}\in\kkp{I}$ with the corresponding (continuous) linear functional $\sum_{i\in I}a_i\gd_i$ (where~$\gd_i$ denotes the $i$th projection), thus justifying such notations as $\Bool(\kkp{I},\kkp{I})$ and $\Op(\kkp{I},\kkp{I})$;
observe that in those notations, the first (resp., second) occurrence of~$\kkp{I}$ is endowed with its structure of \emph{left} (resp., \emph{right}) vector space over~$\kk$.
Moreover, in its second occurrence, $\kkp{I}$ is endowed with the coarsest topology making all canonical projections~$\gd_i$ continuous.

Denote by~$\FL(I,\kk)$ the free left%
\footnote{
``Right'' and ``left'' appear to have been unfortunately mixed up at various places in~\cite{RAlg}, particularly on pages~12 and~13.
Since this is mostly a matter of choosing sides, that paper's results are unaffected.
We nonetheless attempt to fix this here.
}
$\kk$-vector lattice on a set~$I$.
As observed in Baker~\cite{Baker1968}, Bernau~\cite{Bern1969}, Madden \cite[Ch.~III]{MaddTh} (see also Wehrung \cite[page~13]{RAlg} for a summary), $\FL(I,\kk)$ canonically embeds into~$\kk^{\kkp{I}}$.
We sum up a few related facts.

\begin{lemma}[Folklore]\label{L:BakBey}\hfill
\begin{enumerater}
\item\label{FlIkk}
$\FL(I,\kk)$ is isomorphic to the sublattice of $\kk^{\kkp{I}}$ generated by all linear functionals $\sum_{i\in I}a_i\gd_i$ associated to elements $a\in\kkp{I}$, via the assignment $i\mapsto\gd_i$.

\item\label{IdcFlIkk}
The assignment $\seq{x}\mapsto\bck{x\neq0}$ defines an isomorphism from the lattice~$\Idc\FL(I,\kk)$, of all principal \lidl{s} of the left $\kk$-vector lattice~$\FL(I,\kk)$, onto $\Ops(\kkp{I},\kkp{I})$.
\end{enumerater}
\end{lemma}

\section{Consonance kernels}\label{S:ConsKer}

In this section we introduce a tool, the \emph{consonance kernels}, expressing the consonance of the image a lattice homomorphism \emph{via} its behavior on \jirr\ elements.

\begin{definition}\label{D:ConsKer}
Let~$D$ and~$L$ be distributive lattices, with~$D$ finite and~$L$ with a zero element, and let $f\colon D\to L$ be a \jh.
Set $P\eqdef\Ji{D}$.
A \emph{consonance kernel} for~$f$ is a family  $\vecm{e_p}{p\in P}$ of elements of~$L$ such that
 \begin{align}
 f(p)=f(p_*)\vee e_p\,,&\qquad\text{whenever }p\in P\,;\label{Eq:pp*+ep}\\
 e_p\wedge e_q=0\,,&\qquad\text{whenever }p,q\in P
 \text{ are incomparable}\,.\label{Eq:Incomp2Orth}
 \end{align}
We then set $x\sdf_{\vec{e}}y\eqdef\bigvee\setm{e_p}{p\in(P\dnw x)\setminus(P\dnw y)}$, whenever $x,y\in D$.
\end{definition}

\begin{lemma}\label{L:CKDiffOp}
In the context of Definition~\textup{\ref{D:ConsKer}}, $f(x)=f(x\wedge y)\vee(x\sdf_{\vec{e}}y)$ whenever $x,y\in D$.
Moreover, $f$ is a lattice homomorphism.
\end{lemma}

\begin{proof}
Setting $c\eqdef f(x\wedge y)\vee(x\sdf_{\vec{e}}y)$, it is obvious that $c\leq f(x)$.
In order to prove that $f(x)\leq c$, it suffices to prove that $f(p)\leq c$ whenever $p\in P\dnw x$.
By way of contradiction, let $p$ be a minimal element of $P\dnw x$ with $f(p)\nleq c$.
Since $p\leq y$ implies $f(p)\leq f(x\wedge y)\leq c$, we get $p\in (P\dnw x)\setminus(P\dnw y)$, so $f(p)=f(p_*)\vee e_p$.
Since $e_p\leq c$, we get $f(p_*)\nleq c$. The case $p_*=0$ is impossible, because $f(0)\leq f(x\wedge y)\leq c$.
Since~$f$ is a \jh, we get $f(p_*)=\bigvee\setm{f(q)}{q\in P\dnw p_*}$.
By the minimality assumption on~$p$, we get $f(q)\leq c$
for every $q\in P\dnw p_*$, hence $f(p_*)\leq c$, a contradiction.

Now let $x,y\in D$.
By the result of the paragraph above, $f(x)=f(x\wedge y)\vee(x\sdf_{\vec{e}}y)$ and $f(y)=f(x\wedge y)\vee(y\sdf_{\vec{e}}x)$.
Due to~\eqref{Eq:Incomp2Orth}, $(x\sdf_{\vec{e}}y)\wedge(y\sdf_{\vec{e}}x)=0$; whence $f(x)\wedge f(y)=f(x\wedge y)$.
\end{proof}

\begin{lemma}\label{L:ConsKer}
Let~$D$ and~$L$ be distributive lattices, with~$D$ finite and~$L$ with a zero element.
Then a lattice homomorphism $f\colon D\to L$ has a consonance kernel if{f} the range of~$f$ is consonant in~$L$.
\end{lemma}

\begin{proof}
Suppose first that the range of~$f$ is consonant in~$L$.
Since~$D$ is finite, there exists a finite $0$-sublattice~$K$ of~$L$, containing~$f[D]$, such that the range of~$f$ is consonant in~$K$.
Setting $e_p\eqdef f(p)\sd_Kf(p_*)$ for each $p\in\Ji{D}$, Condition~\eqref{Eq:pp*+ep} is obviously satisfied.
Let $p,q\in\Ji{D}$ be incomparable.
{F}rom $p\wedge q\leq p_*$ we get
 \[
 e_p=f(p)\sd_Kf(p_*)\leq f(p)\sd_Kf(p\wedge q)=
 f(p)\sd_K(f(p)\wedge f(q))=f(p)\sd_Kf(q)\,,
 \]
and, similarly, $e_q\leq f(q)\sd_Kf(p)$.
Since~$f(p)$ and~$f(q)$ are consonant within~$K$, we get $(f(p)\sd_Kf(q))\wedge(f(q)\sd_Kf(p))=0$; whence $e_p\wedge e_q=0$.

Let, conversely, $\vecm{e_p}{p\in\Ji{D}}$ be a consonance kernel for~$f$ and set $P\eqdef\Ji{D}$.
Let $x,y\in D$, set $u\eqdef x\sdf_{\vec{e}}y$ and $v\eqdef y\sdf_{\vec{e}}x$.
It follows from Lemma~\ref{L:CKDiffOp} that $f(x)\leq f(y)\vee u$ and $f(y)\leq f(x)\vee v$.
Moreover, for all $p\in(P\dnw x)\setminus(P\dnw y)$ and $q\in(P\dnw y)\setminus(P\dnw x)$, $p$ and~$q$ are incomparable, thus $e_p\wedge e_q=0$; whence $u\wedge v=0$.
Therefore, the pair $(u,v)$ witnesses the consonance of~$f(x)$ and~$f(y)$ in~$L$.
\end{proof}

\section{An extension lemma for infinite distributive lattices}
\label{S:ExtLemma}

This section's main result, Lemma~\ref{L:MainExt}, states conditions under which a homomorphism $f\colon D\to L$ of distributive lattices can be extended to a homomorphism $f\colon E\to L$ in case~$E$ is generated over~$D$ by two disjoint elements~$a$ and~$b$.
One of its main improvements, over the original \cite[Lemma~4.2]{MV1} it stems from, is the possibility of~$D$ be infinite.

\begin{definition}\label{D:SemiHeyting}
A $0,1$-sublattice~$D$ of a bounded distributive lattice~$E$ is a \emph{semi-Heyting sublattice} if for all $x,y\in D$, $x\rightarrow_Dy$ and $x\rightarrow_Ey$ both exist and are equal.
\end{definition}

In particular, every semi-Heyting sublattice of~$E$ is a Heyting algebra ($E$ itself may not be a Heyting algebra).

\begin{notation}\label{Not:fv(a)}
Let~$D$ be a finite $0,1$-sublattice of a bounded distributive lattice~$E$ and let $f\colon D\to L$ be a $0$-lattice homomorphism.
We set
 \[
 f_{\vec{e}}(a)\eqdef\bigvee\setm{e_p}{p\in\Ji{D}\,,\ p\leq p_*\vee a}
 \]
for every consonance kernel~$\vec{e}$ of~$f$ and every $a\in E$.
\end{notation}

The following lemma arises from Wehrung \cite[Remark~4.6]{RAlg}.
We include a proof for convenience.

\begin{lemma}\label{L:pqincomp}
Let~$D$ be a finite semi-Heyting sublattice of a bounded distributive lattice~$E$, let $f\colon D\to L$ be a $0$-lattice homomorphism, and let $a,b\in E$ such that $a\wedge b=0$.
Then any \jirr\ elements~$p$ and~$q$ in~$D$ such that $p\leq p_*\vee a$ and $q\leq q_*\vee b$ are incomparable.
In particular, $f_{\vec{e}}(a)\wedge f_{\vec{e}}(b)=0$ for any consonance kernel~$\vec{e}$ for~$f$.
\end{lemma}

\begin{proof}
Suppose otherwise, say $p\leq q$; thus $p^{\dagger}\leq\nobreak q^{\dagger}$.
{F}rom $a\wedge b=0$ we get $p\wedge b\leq(p_*\vee a)\wedge b=p_*\wedge b\leq p_*$, thus, by assumption, $b\leq p\rightarrow_Ep_*=p\rightarrow_{D}p_*=p^{\dagger}$.
Since $p^{\dagger}\leq q^{\dagger}$, we get $b\leq q^{\dagger}$, so $q\leq q_*\vee b\leq q^{\dagger}$, a contradiction.
\end{proof}

We are now reaching this section's main goal. In the next proof we use the following well known extension criterion.
Let~$D$ and~$L$ be distributive lattices and~$X$ a generating subset of~$D$.
Then a map $f\colon X\to L$ can be
extended to a (necessarily unique) homomorphism $g\colon D\to L$ if and only if
\begin{equation}\bigwedge_{i=1}^mx_i\leq\bigvee_{j=1}^ny_j\quad\Longrightarrow\quad
\bigwedge_{i=1}^mf(x_i)\leq\bigvee_{j=1}^nf(y_j)\label{extcrit}
\end{equation}
for all $m,n>0$ and all $x_1,\dots,x_m,y_1,\dots,y_n\in X$.

\begin{lemma}[Main Extension Lemma]\label{L:MainExt}
Let~$D$ be a semi-Heyting sublattice of a bounded distributive lattice~$E$ and let $a,b\in E$.
Setting $B\eqdef\BR(D)$, we assume the following:
\begin{enumerater}
\item\label{E=Dab}
$E$ is generated, as a lattice, by~$D\cup\set{a,b}$.

\item\label{ammb=0}
$a\wedge b=0$.

\item\label{abcBD}
All elements~$a_B$, $b_B$, $(a\vee b)_B$, $a^B$, and~$b^B$ are defined.

\item\label{aveebB}
$(a\vee b)_B=a_B\vee b_B$.

\end{enumerater}

Then \begin{equation}\label{Eq:c^BDD'}
 c^B\in D\text{ whenever }c\in\set{a,b,a\vee b}\,.
 \end{equation}
Further, for every $0$-lattice homomorphism $f\colon D\to L$ and all $\ga,\gb\in L$, the following
conditions are equivalent.
\begin{enumeratei}
\item $(\ga,\gb)=(g(a),g(b))$ for some lattice homomorphism $g\colon E\to L$ extending~$f$;
\item  $\ga\leq f(a^B)$, $\gb\leq f(b^B)$, $\ga\wedge\gb=0$, $\BR(f)\pI{a_B}\leq\ga$, and $\BR(f)\pI{b_B}\leq\gb$.
\end{enumeratei}

Moreover, for  any finite semi-Heyting sublattice~$D'$ of~$D$ such that $\set{a_B,b_B}\subseteq\BR(D')$ and $\set{a^B,b^B}\subseteq D'$, and any consonance kernel~$\vec{e}$ of~$f'\eqdef f\res_{D'}$, $(f'_{\vec{e}}(a),f'_{\vec{e}}(b))$ is a pair satisfying (ii).

%
%
%
%
%
%

\end{lemma}

\begin{note}
By the same token as the one used in the proof of Lemma~\ref{L:BR2Ineq}, the condition that $\BR(f)(a_B)\leq\ga$ is equivalent to saying that for all $x,y\in D$, $x\leq y\vee a\Rightarrow f(x)\leq f(y)\vee\ga$.
By Lemma~\ref{L:BR2Ineq}, if~$D$ is finite, then it suffices to restrict ourselves to the case where $x=p\in\Ji{D}$ and $y=p_*$.
Note that~$\BR(f)(a_B)$ is an element of~$\BR(L)$, usually not in~$L$, so it cannot be taken as the lowest possible value of~$\ga$ \emph{a priori}.
\end{note}

\begin{proof}
We start by proving \eqref{Eq:c^BDD'}.
By~\eqref{abcBD}, there is an expression of the form $c^B=\bigwedge_{i<n}(\neg u_i\vee v_i)$ (within~$B$) where $n<\go$ and all $u_i,v_i\in D$.
For each $i<n$, $c\leq\neg u_i\vee v_i$ within~$\BR(E)$, thus $u_i\wedge c\leq v_i$, and thus, since~$D$ is a semi-Heyting sublattice of~$E$, $c\leq u_i\rightarrow_Ev_i=u_i\rightarrow_{D}v_i$; whence, setting $w\eqdef\bigwedge_{i<n}(u_i\rightarrow_{D}v_i)$, we get $c\leq w$.
For each $i<n$, $w\leq u_i\rightarrow_{D}v_i$ with $w\in D$, thus $u_i\wedge w\leq v_i$, so $w\leq\neg u_i\vee v_i$ within~$B$, and so $w\leq c^B$.
Since $w\in D$, it follows that $w=c^B=c^D$.

Now it is obvious that for every lattice homomorphism $g\colon E\to L$ extending~$f$, the pair $(\ga,\gb)\eqdef(g(a),g(b))$ satisfies $\ga\leq f(a^B)$, $\gb\leq f(b^B)$, $\ga\wedge\gb=0$, $\BR(f)\pI{a_B}\leq\ga$, and $\BR(f)\pI{b_B}\leq\gb$.
Let, conversely, $(\ga,\gb)$ be such a pair.
\setcounter{claim}{0}

We need to show the implication \eqref{extcrit} for $x_i,y_j\in D\cup\{a,b\}$.
Since~$f$ is a lattice homomorphism,
we can assume that exactly one~$x_i$ and exactly one~$y_j$ belong to~$D$.
Since $a\wedge b=0$, the inequality
$x\wedge a\leq y\vee b$ is equivalent to $x\wedge a\leq y$.
So, \eqref{extcrit} boils down to the equality
$\alpha\wedge\beta=0$ (which is assumed) and the following
implications:
 \begin{align}
 x\leq y\vee a&\Rightarrow f(x)\leq f(y)\vee\ga\,;
 \label{Eq:xleqy+a}\\
 x\leq y\vee b&\Rightarrow f(x)\leq f(y)\vee\gb\,;
 \label{Eq:xleqy+b}\\
 x\leq y\vee a\vee b&\Rightarrow f(x)\leq
 f(y)\vee\ga\vee\gb\,;\label{Eq:xleqy+ab}\\
 x\wedge a\leq y&\Rightarrow f(x)\wedge\ga\leq f(y)\,;
 \label{Eq:xaleqy}\\
 x\wedge b\leq y&\Rightarrow f(x)\wedge\gb\leq f(y)\,.
 \label{Eq:xbleqy}
 \end{align}
The implications~\eqref{Eq:xleqy+a} and~\eqref{Eq:xleqy+b} follow from $\BR(f)\pI{a_B}\leq\ga$ and $\BR(f)\pI{b_B}\leq\nobreak\gb$.
Owing to Condition~\eqref{aveebB}, the implication~\eqref{Eq:xleqy+ab} follows from the inequalities
 \[
 \BR(f)\pI{(a\vee b)_B}=\BR(f)\pI{a_B\vee b_B}=
 \BR(f)(a_B)\vee\BR(f)(b_B)\leq\ga\vee\gb\,.
 \]
Suppose that $x\wedge a\leq y$.
Since~$D$ is a semi-Heyting sublattice of~$E$, it follows that $a\leq x\rightarrow_Ey=x\rightarrow_Dy$, thus, using~\eqref{Eq:c^BDD'}, $a^D=a^B\leq x\rightarrow_Dy$.
It follows that $\ga\leq f(a^B)\leq f(x\rightarrow_Dy)$, thus $f(x)\wedge\ga\leq f(x)\wedge f(x\rightarrow_Dy)\leq f(y)$.
The implication~\eqref{Eq:xaleqy} follows.
The proof of~\eqref{Eq:xbleqy} is similar.

For the remainder of the proof, let~$D'$ be a finite semi-Heyting sublattice of~$D$ such that $\set{a_B,b_B}\subseteq\BR(D')$ and $\set{a^B,b^B}\subseteq D'$ (cf. Figure~\ref{Fig:MainExt}), and let~$\vec{e}$ be a consonance kernel of $f'\eqdef f\res_{D'}$.
Set $(\ga,\gb)\eqdef(f'_{\vec{e}}(a),f'_{\vec{e}}(b))$.

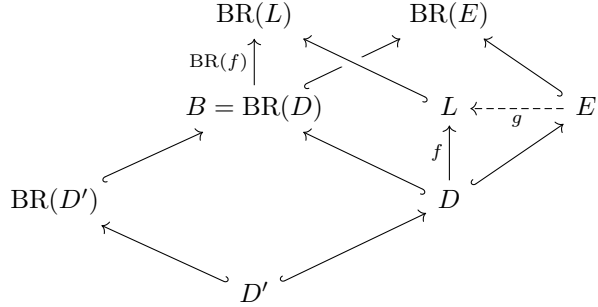
\begin{figure}[htb]
\begin{tikzcd}
\centering
& \BR(L) & \BR(E) &\\
& B=\BR(D)\arrow[u,"\BR(f)"]\arrow[ur,hook'] &
L\arrow[ul,hook,crossing over] & E\arrow[ul,hook]\arrow[l,"g",dashed]\\
\ \BR(D')\arrow[ur,hook'] &&
D\arrow[ul,hook]\arrow[u,"f"]\arrow[ur,hook] &\\
& D'\arrow[ul,hook']\arrow[ur,hook] &&
\end{tikzcd}
\caption{Illustrating the proof of Lemma~\ref{L:MainExt}}
\label{Fig:MainExt}
\end{figure}

For every $p\in\Ji{D'}$, $p\leq p_*\vee a$ (within~$E$) implies that $p\leq p_*\vee a^B$ (within~$D'$), thus, since $p\in\Ji{D'}$, we get $p\leq a^B$, whence $e_p\leq f(p)\leq f(a^B)$.
This proves that $\ga\leq f(a^B)$.
Similarly, $\gb\leq f(b^B)$.
Further, the equation $\ga\wedge\gb=0$ follows from Lemma~\ref{L:pqincomp}.

Let $c\in\set{a,b}$ and let $x,y\in D$ such that $x\leq y\vee c$, we need to prove that $f(x)\leq f(y)\vee f'_{\vec{e}}(c)$.
{F}rom $x\wedge\neg y\leq c$ (within~$\BR(E)$) it follows that $x\wedge\neg y\leq c_B$ (within~$\BR(D)$).
Set $X\eqdef\setm{p\in\Ji{D'}}{p\wedge\neg p_*\leq c_B}=\setm{p\in\Ji{D'}}{p\leq p_*\vee c}$.
By~\eqref{abcBD} and since~$\BR(D')$ is a finite Boolean algebra with atoms~$p\wedge\neg p_*$ for $p\in\Ji{D'}$,
 \begin{equation}\label{Eq:Decc_B}
 c_B=\bigvee\setm{p\wedge\neg p_*}{p\in X}\quad
 \text{within }B\,.
 \end{equation}
By the definition of~$X$,
 \[
 f(p)=f(p_*)\vee e_p\leq f(p_*)\vee f'_{\vec{e}}(c)\quad
 \text{whenever }p\in X\,,
 \]
so $f(p)\wedge\neg f(p_*)\leq f'_{\vec{e}}(c)$ within~$\BR(L)$, whenever $p\in X$; whence, using~\eqref{Eq:Decc_B},
 \begin{equation}\label{Eq:BRfcBRD}
 \BR(f)(c_B)=\bigvee\setm{f(p)\wedge\neg f(p_*)}{p\in X}
 \leq f'_{\vec{e}}(c)
 \quad\text{within }\BR(L)\,.
 \end{equation}
Using~\eqref{Eq:BRfcBRD}, we get
 \[
 f(x)\wedge\neg f(y)=\BR(f)(x\wedge\neg y)\leq\BR(f)(c_B)\leq
 f'_{\vec{e}}(c)
 \,,
 \]
so $f(x)\leq f(y)\vee f'_{\vec{e}}(c)$.
\end{proof}

\section{Adjunctions between lattices~$\Bool(\cF,\kkp{I})$}\label{S:ExtOp}

Throughout this section~$\kk$ will be a totally ordered division ring.
In this section we shall state a few properties of Boolean algebras of the form $\Bool(\cF,\gO)$, mostly related to relative completeness between such algebras.

The following observation is contained in the proof of Wehrung \cite[Lemma~6.6]{MV1}.

\begin{lemma}\label{L:Ca+b}
Let~$\gO$ be a convex subset in a right vector space~$\EE$ over~$\kk$ and let~$\cF\cup\nobreak\set{a}$ be a set of affine functionals on~$\EE$.
Set $A^+\eqdef\bck{a>0}$ and $A^-\eqdef\bck{a<0}$.
Then for every $U\in\Bool(\cF,\gO)$, if $U\subseteq A^+\cup A^-$, then there are $U^+,U^-\in\Bool(\cF,\gO)$ such that $U=U^+\cup U^-$ whereas $U^+\subseteq A^+$ and $U^-\subseteq A^-$.
\end{lemma}

\begin{proof}
Since~$U$ is the union of finitely many cells, each of which being the intersection of finitely many sets of the form either $\bck{\pm f>0}$ or $\bck{\pm f\geq0}$ where $f\in\cF$, it suffices to consider the case where~$U$ is such a cell.
If~$U$ meets both~$A^+$ and~$A^-$, pick $x\in U\cap A^+$ and $y\in U\cap A^-$; so $a(x)>0$ and $a(y)<0$.
Then $\gl\eqdef(a(y)-a(x))^{-1}a(y)$ belongs to the open interval~$\oo{0,1}$ and $a(x\gl+y(1-\gl))=0$, that is, $x\gl+y(1-\gl)\notin A^+\cup A^-$.
On the other hand, since~$U$ is convex, $x\gl+y(1-\gl)\in U$; a contradiction since $U\subseteq A^+\cup A^-$.
Therefore, $U$ is disjoint either from~$A^+$ or from~$A^-$, thus it is contained either in~$A^+$ or in~$A^-$.
\end{proof}

\begin{corollary}\label{C:Ca+b}
In the context of Lemma~\textup{\ref{C:Ca+b}}, $(A^+\cup A^-)_{\Bool(\cF,\gO)}$ exists if{f} both $(A^+)_{\Bool(\cF,\gO)}$ and $(A^-)_{\Bool(\cF,\gO)}$  exist, and then
 \[
 (A^+\cup A^-)_{\Bool(\cF,\gO)}=
 (A^+)_{\Bool(\cF,\gO)}\cup(A^-)_{\Bool(\cF,\gO)}\,.
 \]
\end{corollary}

In what follows we will identify every element $f\in\kkp{I}$ with the associated linear functional on~$\kkp{I}$, that is, $x\mapsto\sum_{i\in I}f_ix_i$.
Moreover, whenever $I\subseteq J$ we will identify~$\kkp{I}$ with the subset of~$\kkp{J}$ consisting of all vectors with support contained in~$I$.

\begin{notation}\label{Not:BoolkI}
For $I\subseteq J$, we define mappings 
 \begin{align*}
 \eps_{I,J}&\colon\Pow\kkp{I}\to\Pow\kkp{J}\,,\\
 \rho_{J,I}^{\vee},\rho_{J,I}^{\wedge}&\colon\Pow\kkp{J}\to\Pow\kkp{I}\,,
 \end{align*}
by
\begin{align*}
\eps_{I,J}(X)&=\setm{y\in\kkp{J}}{y\res_I\in X}\,,\\
 \rho_{J,I}^{\wedge}(Y)&=\setm{x\in\kkp{I}}
 {(\forall y\in\kkp{J})(y\res_I=x\Rightarrow y\in Y)}\,,\\
 \rho_{J,I}^{\vee}(Y)&=\setm{y\res_I}{y\in Y}\,. 
 \end{align*}
\end{notation}

The following statements are immediate consequences of the definitions:
\begin{itemize}
\item 
$\eps_{I,J}$ is an embedding of Boolean algebras, $\rho_{J,I}^{\wedge}$ is a \mh, and $\rho_{J,I}^{\vee}$ is a \jh.
Moreover, $\rho_{J,I}^{\wedge}$ and $\rho_{J,I}^{\vee}$ are right and left adjoint to~$\eps_{I,J}$, respectively.

\item
$\rho_{J,I}^{\wedge}$ and~$\rho_{J,I}^{\vee}$ are conjugate, that is, $\kkp{I}\setminus\rho_{J,I}^{\wedge}(Y)=\rho_{J,I}^{\vee}\pII{\kkp{J}\setminus Y}$ whenever $Y\subseteq\kkp{J}$.
\end{itemize}

\begin{lemma}\label{L:gegrPresBool}
Let~$I$ and~$J$ be sets with $I\subseteq J$.
The following statements hold:
\begin{enumerater}
\item\label{relcomp} 
$\eps_{I,J}\rho_{J,I}^{\wedge}(Z)=\eps_{I,J}\rho_{J,I}^{\wedge}(Z)=Z$ for every $Z\in\Bool(\kkp{I},\kkp{J})$.

\item\label{eps2Bool}
$\eps_{I,J}[\Bool(\kkp{I},\kkp{I})]=\Bool(\kkp{I},\kkp{J})\subseteq\Bool(\kkp{J},\kkp{J})$.

\item\label{rho2Bool}
$\rho_{J,I}^{\wedge}[\Bool(\kkp{J},\kkp{J})]=\rho_{J,I}^{\vee}[\Bool(\kkp{J},\kkp{J})]=\Bool(\kkp{I},\kkp{I})$.

\end{enumerater}
\end{lemma}

\begin{proof}
\emph{Ad}~\eqref{relcomp} and \eqref{eps2Bool} are both trivial.
In order to prove~\eqref{rho2Bool}, it suffices, since~$\rho_{J,I}^{\wedge}$ and~$\rho_{J,I}^{\vee}$ are conjugate, to establish the result for~$\rho_{J,I}^{\vee}$.
For every $X\in\Bool(\kkp{I},\kkp{I})$, $X=\rho_{J,I}^{\vee}\eps_{I,J}(X)$ with $\eps_{I,J}(X)\in\Bool(\kkp{J},\kkp{J})$, thus $X\in\rho_{J,I}^{\vee}[\Bool(\kkp{J},\kkp{J})]$; whence
$\rho_{J,I}^{\vee}[\Bool(\kkp{J},\kkp{J})]$ contains~$\Bool(\kkp{I},\kkp{I})$.

Let us establish the converse containment.
Since~$\rho_{J,I}^{\vee}$ is a \jzh, it suffices to prove that $\rho_{J,I}^{\vee}(Y)\in\Bool(\kkp{I},\kkp{I})$ whenever~$Y$ is a set of the form $\bigcap_{i<m}\bck{a_i\geq0}\cap\bigcap_{j<n}\bck{b_j>0}$ where $m,n<\go$ and all $a_i,b_j\in\kkp{J}$.

Set $a'_i\eqdef a_i\res_I$ and $a''_i\eqdef a_i\res_{J\setminus I}$, for all $i<m$, and define similarly~$b'_j$ and~$b''_j$ for $j<n$.
An element $x\in\kkp{I}$ belongs to $\rho_{J,I}^{\vee}(Y)$ if{f} there exists $z\in\kkp{J\setminus I}$ such that each $a'_i(x)+a''_i(z)\geq0$ and each $b'_j(x)+b''_j(z)>0$.
The set~$V$ of all $(m+n)$-tuples of elements of~$\kk$ of the form $(a''_0(z),\dots,a''_{m-1}(z),b''_0(z),\dots,b''_{n-1}(z))$ is a vector subspace of~$\kk^{m+n}$.
Hence, an element $x\in\kkp{I}$ belongs to $\rho_{J,I}^{\vee}(Y)$ if{f} there exists $u\in V$ such that $a'_i(x)+u_i\geq0$ whenever $i<m$ and $b'_j(x)+u_{m+j}>0$ whenever $j<n$.
Since membership in~$V$, of any $(m+n)$-tuple of elements of~$\kk$, can be expressed by a finite set of linear equations, the statement that a given $x\in\nobreak\kkp{I}$ belongs to $\rho_{J,I}^{\vee}(Y)$ can be expressed by a sentence, over the first-order language $\cL\eqdef\set{<,0,-,+}\cup\setm{\cdot\gl}{\gl\in\kk}$ of ordered Abelian groups augmented with right scalar multiplications by elements of~$\kk$, in $(a'_0(x),\dots,a'_{m-1}(x),b'_0(x),\dots,b'_{n-1}(x))$.
Now every $\cL$-sentence is equivalent, over all nonzero totally ordered right $\kk$-vector spaces, to a quantifier-free $\cL$-sentence (cf. van den Dries \cite[Corollary~I.7.8]{vdD1998}).
Therefore, $\rho_{J,I}^{\vee}(Y)$ belongs to $\Bool(\cF,\kkp{I})$ for a finite set~$\cF$ of linear combinations of the~$a'_i$ and the~$b'_j$.
\end{proof}

\begin{proposition}\label{P:gegrPresBool}
Let~$I$ and~$J$ be sets with $I\subseteq J$ and let~$\cD$ be a finite subset of~$\kkp{J}$.
Then $\Bool(\kkp{I}\cup\cD,\kkp{J})$ is relatively complete in~$\Bool(\kkp{J},\kkp{J})$.
\end{proposition}

\begin{proof}
We first prove that $\Bool(\kkp{I},\kkp{J})$ is relatively complete in 
$\Bool(\kkp{J},\kkp{J})$.
Let $Y\in\Bool(\kkp{J},\kkp{J})$.
By Lemma~\ref{L:gegrPresBool}, 
$\eps_{I,J}\rho^{\vee}_{J,I}(Y), \eps_{I,J}\rho^{\wedge}_{J,I}(Y)\in\Bool(\kkp{I},\kkp{J})$. Further,
$Y\subseteq Z\in \Bool(\kkp{I},\kkp{J})$ implies $\eps_{I,J}\rho^{\vee}_{J,I}(Y)\subseteq \eps_{I,J}\rho^{\vee}_{J,I}(Z)=Z$.
Thus, $Y^{\Bool(\kkp{I},\kkp{J})}=\eps_{I,J}\rho^{\vee}_{J,I}(Y)$ and similarly, $Y_{\Bool(\kkp{I},\kkp{J})}=\eps_{I,J}\rho^{\wedge}_{J,I}(Y)$.

Since $\Bool(\kkp{I}\cup\cD,\kkp{J})$ is finitely generated over $\Bool(\kkp{I},\kkp{J})$ (via the additional generators~$\bck{d>0}$ and~$\bck{d<0}$ for $d\in\cD$),
the desired conclusion follows from Lemma~\ref{L:RelCplBA}.
\end{proof}

\section{Extending a top-faithful map}\label{S:ExtTF}

In Lemmas~\ref{L:DomStep} and~\ref{L:RngStep} we fix a totally ordered division ring~$\kk$.
The following lemma takes care of the ``domain step'' required in the proof of Theorem~\ref{T:Al1hom}.

\begin{lemma}\label{L:DomStep}
Let~$I$ and~$J$ be sets, let~$L$ be a completely normal distributive $0$-lattice, let~$\cD$ be a finite subset of~$\kkp{J}$, and let $e\in\kkp{J}$.
Then every top-faithful $0$-lattice homomorphism $f\colon\Op(\kkp{I}\cup\cD,\kkp{J})\to\ip{L}$ extends to a top-faithful lattice homomorphism $g\colon\Op(\kkp{I}\cup\cD\cup\set{e},\kkp{J})\to\ip{L}$ \pup{cf. Figure~\textup{\ref{Fig:DomStep}}}.
\end{lemma}

\begin{figure}[htb]
\begin{tikzcd}
\centering
\ip{L} & \\
\Op(\kkp{I}\cup\cD,\kkp{J})\arrow[u,"f"]\arrow[r,"\subseteq",hook] &
\Op(\kkp{I}\cup\cD\cup\set{e},\kkp{J}))\arrow[ul,"g"',dashed]
\end{tikzcd}
\caption{A commutative triangle for Lemma~\ref{L:DomStep}}
\label{Fig:DomStep}
\end{figure}

\begin{proof}
Set $\cE\eqdef\cD\cup\set{e}$, $D\eqdef\Op(\kkp{I}\cup\cD,\kkp{J})$, $E\eqdef\Op(\kkp{I}\cup\cE,\kkp{J})$, $B\eqdef\BR(D)=\Bool(\kkp{I}\cup\cD,\kkp{J})$, and $C\eqdef\Bool(\kkp{I}\cup\cE,\kkp{J})$.
By Proposition~\ref{P:gegrPresBool}, $B$ is relatively complete in~$C$.
In particular, setting $a\eqdef\bck{e>0}$ and $b\eqdef\bck{e<0}$, the elements~$a^B$, $b^B$, $a_B$, $b_B$, and $(a\vee b)_B$ are all defined.
By Corollary~\ref{C:Ca+b}, $(a\vee b)_B=a_B\vee b_B$.
Let~$\cD'$ be a finite subset of $\kkp{I}\cup\cD$ such that~$a^B$, $b^B$, $a_B$, and~$b_B$ all belong to $B'\eqdef\Bool(\cD',\kkp{J})$.
By Wehrung \cite[Lemma~5.4]{MV1} (see also Wehrung \cite[Lemma~4.1]{RAlg} for the more general form of that statement), $D$ is a Heyting subalgebra of~$E$ and 
$D'\eqdef\Op(\cD',\kkp{J})$ is a Heyting subalgebra of~$D$.
Since~$L$ is completely normal and~$f[D']$ is finite, it follows from Lemma~\ref{L:ConsKer} that $f'\eqdef f\res_{D'}$ has a consonance kernel $\vecm{e_P}{P\in\Ji{D'}}$.
By Lemma~\ref{L:MainExt}, $f$ extends to a unique lattice homomorphism $g\colon D\to L$ such that $g(x)=f'_{\vec{e}}(x)$ whenever $x\in\set{a,b}$.
For any $P\in\Ji{D'}$ such that $P\subseteq P_*\cup x$, $0\notin P_*\cup x$, thus $0\notin P$, that is, $P$ is not the top element of~$\Op(\kkp{I},\kkp{J})$.
Since~$f$ is top-faithful, it follows that $e_P\leq f(P)<\infty$; whence $f'_{\vec{e}}(x)<\infty$.
It follows that~$g$ is top-faithful.
\end{proof}

The ``surjectivity step'' is much easily taken care of:

\begin{lemma}\label{L:RngStep}
Let~$I$ and~$J$ be sets with $I\subset J$ and $J\setminus I$ infinite, let~$L$ be a distributive $0$-lattice, let~$\cD$ be a finite subset of~$\kkp{J}$, and let $c\in L$.
Then every for every top-faithful $0$-lattice homomorphism $f\colon\Op(\kkp{I}\cup\cD,\kkp{J})\to\ip{L}$, there are $e\in\kkp{J}$ and a top-faithful lattice homomorphism $g\colon\Op(\kkp{I}\cup\cD\cup\set{e},\kkp{J})\to\ip{L}$ such that $g(e)=c$.
\end{lemma}

\begin{proof}
Since~$\cD$ is finite and $J\setminus I$ is infinite, there exists $j\in J\setminus I$ not in the support of any element of~$\cD$.
Take $e\eqdef\gd_j$, the $j$th canonical projection $\kkp{J}\twoheadrightarrow\kk$.
By the argument of Wehrung \cite[Lemma~8.3]{MV1}, $\Op(\kkp{I}\cup\cD\cup\set{\gd_j},\kkp{J})$ is the (internal) free amalgamated sum of $\Op(\kkp{I}\cup\cD,\kkp{J})$ and $\set{\es,\bck{\gd_j>0},\bck{\gd_j<0},\bck{\gd_j\neq0},\kkp{J}}$ within the category of bounded distributive lattices.
Hence~$f$ extends to a unique lattice homomorphism $g\colon\Op(\kkp{I}\cup\cD\cup\set{\gd_j},\kkp{J})\to L$ such that $g(\bck{\gd_j>0})=c$ and $g(\bck{\gd_j<0})=0$.
Since $c<\infty$ and~$f$ is top-faithful, it follows that~$g$ is also top-faithful.
\end{proof}

\section{Representing trees of countable lattices}\label{S:ReprTrees}

In this section we will reach the paper's main goal, Theorem~\ref{T:Al1hom}, which states that every completely normal distributive $0$-lattice is a homomorphic image of some~$\Idc{F}$ for some $\kk$-vector lattice~$F$.
In order to reach that result we will in fact prove (cf. Theorem~\ref{T:Part3Ext}) the apparently stronger statement that every \emph{diagram} of countable completely normal distributive $0$-lattices, indexed by a tree in which every element has countable height, can be represented in that fashion.

Towards that goal, our main technical tool is the following ``one-step extension'' theorem, which relies on the results of Section~\ref{S:ExtTF}, together with the observation that for $\cF\subseteq\kkp{I}$, $\Op(\cF,\kkp{I})=\Ops(\cF,\kkp{I})\sqcup\set{\infty}$ (where~$\infty$ denotes here the full space~$\kkp{I}$; so the top-faithful maps $\Op(\cF,\kkp{I})\to\ip{L}$ are exactly the~$\ip{g}$ where $g\colon\Ops(\cF,\kkp{I})\to L$).

\begin{theorem}\label{T:Part1Ext}
Let~$\kk$ be a countable totally ordered division ring, let~$I$ and~$J$ be countable sets with $I\subset J$ and~$J\setminus I$ infinite, let~$K$ and~$L$ be distributive $0$-lattices with~$L$ countable completely normal, let $\gf\colon K\to L$ be a $0$-lattice homomorphism, and let $f\colon\Ops(\kkp{I},\kkp{I})\to K$ be a $0$-lattice homomorphism.
Then there exists a surjective lattice homomorphism $g\colon\Ops(\kkp{J},\kkp{J})\twoheadrightarrow L$ such that $g\circ\eps_{I,J}=\gf\circ f$.
\end{theorem}

The settings for Theorem~\ref{T:Part1Ext} can be read on Figure~\ref{Fig:Part1Ext}.
Its proof can be followed on Figure~\ref{Fig:Part1Ext2}.

\begin{figure}[htb]
\begin{tikzcd}
\centering
K\arrow[r,"\gf"] & L \\
\Ops(\kkp{I},\kkp{I})\arrow[u,"f"]\arrow[r,"\eps_{I,J}",hook] &
\Ops(\kkp{J},\kkp{J})\arrow[u,"g"',dashed,twoheadrightarrow]
\end{tikzcd}
\caption{A commutative diagram for Theorem~\ref{T:Part1Ext}}
\label{Fig:Part1Ext}
\end{figure}

\begin{figure}[htb]
\begin{tikzcd}
\centering
K\arrow[r,"\gf"] & L & \\
\Ops(\kkp{I},\kkp{I})\arrow[u,"f"]\arrow[r,"\eps_{I,J}",hook] &
\Ops(\kkp{I}\cup\cD_n,\kkp{J})\arrow[u,"g_n"]
\arrow[r,"\subseteq",hook] &
\Ops(\kkp{I}\cup\cD_{n+1},\kkp{J})\arrow[ul,"g_{n+1}"] 
\end{tikzcd}
\caption{Illustrating the proof of Theorem~\ref{T:Part1Ext}}
\label{Fig:Part1Ext2}
\end{figure}

\begin{proof}
An iterative application of Lemmas~\ref{L:DomStep} and~\ref{L:RngStep}, similar to the proof of Wehrung \cite[Theorem~9.1]{MV1} but easier since we do not need any analogue of the ``closure step" \cite[Lemma~7.1]{MV1}.
Let $\kkp{J}=\setm{v_n}{n<\go}$ and $L=\setm{c_n}{n<\go}$.
Given an extension $g_n\colon\Ops(\kkp{I}\cup\cD_n,\kkp{J})\to L$ of $g_0\eqdef\gf\circ f$, where $\cD_n\subset\kkp{J}$ is finite, we extend the top-faithful extension $\ip{g}_n\colon\Op(\kkp{I}\cup\cD_n,\kkp{J})\to\ip{L}$ of~$g_n$ to a top-faithful lattice homomorphism $\ip{g}_{n+1}\colon\Op(\kkp{I}\cup\cD_{n+1},\kkp{J})\to\ip{L}$, with $\cD_n\subseteq\cD_{n+1}$, $v_{\lfloor n/2\rfloor}\in\cD_{n+1}$ if~$n$ is even (via Lemma~\ref{L:DomStep}), and $c_{\lfloor n/2\rfloor}\in\rng g_{n+1}$ if~$n$ is odd (via Lemma~\ref{L:RngStep}).
The common extension~$g$ of all~$g_n$ is as required.
\end{proof}

By virtue of Lemma~\ref{L:BakBey}, Theorem~\ref{T:Part1Ext} can be recast in terms of \lidl\ lattices of free vector lattices over~$\kk$, as follows.

\begin{theorem}\label{T:Part2Ext}
Let~$\kk$ be a countable totally ordered division ring, let~$I$ and~$J$ be countable sets with $I\subset J$ and~$J\setminus I$ infinite, let~$K$ and~$L$ be distributive $0$-lattices with~$L$ countable completely normal, let $\gf\colon K\to L$ be a $0$-lattice homomorphism, and let $f\colon\Idc\FL(I,\kk)\to K$ be a $0$-lattice homomorphism.
Denote by $\eta_{I,J}\colon\Idc\FL(I,\kk)\hookrightarrow\Idc\FL(J,\kk)$ the canonical embedding.
Then there exists a surjective lattice homomorphism $g\colon\Idc\FL(J,\kk)\twoheadrightarrow L$ such that $g\circ\eta_{I,J}=\gf\circ f$.
\end{theorem}

By using the functoriality of the assignment $I\mapsto\Idc\FL(I,\kk)$, Theorem~\ref{T:Part2Ext} can further be extended to diagrams indexed by trees, as follows.

\begin{theorem}\label{T:Part3Ext}
Let~$\kk$ be a countable totally ordered division ring, let~$T$ be a tree in which every element has countable height, and let $\vec{L}\eqdef\vecm{L_s,\gf_{s,t}}{s\leq t\text{ in }T}$ be a commutative $T$-indexed diagram of distributive $0$-lattices such that~$L_t$ is countable completely normal whenever $t\in T\setminus\set{\bot}$.
Let $I_{\bot}\subseteq\set{\bot}\times\go$ and set\linebreak $I_t\eqdef(T\dnw t)\times\go$ whenever $t\in T\setminus\set{\bot}$.
Set $\vec{I}\eqdef\vecm{I_s,\eta_{I_s,I_t}}{s\leq t\text{ in }T}$.
Then every $0$-lattice homomorphism $\chi_{\bot}\colon\Idc\FL(I_{\bot},\kk)\to L_{\bot}$ extends to a natural transformation $\vec{\chi}\colon\Idc\FL(\vec{I},\kk)\todot\vec{L}$ such that~$\chi_t$ is a surjective lattice homomorphism whenever $t\in T\setminus\set{\bot}$.
\end{theorem}

\begin{proof}
The proof can be partly followed on Figure~\ref{Fig:Part3Ext}.
\begin{figure}[htb]
\begin{tikzcd}
\centering
L_s\arrow[r,"\gf_{s,<t}","\varinjlim"']\arrow[rr,"\gf_{s,t}"',bend left] &
L_{<t}\arrow[r,"\gf_{<t,t}"] & L_t\\
\Idc\FL(I_s,\kk)\arrow[u,"\chi_s"]
\arrow[r,"\eta_{I_s,I_{<t}}"',"\varinjlim",hook]
\arrow[rr,"\eta_{I_s,I_t}",bend right,hook] & \Idc\FL(I_{<t},\kk)
\arrow[u,"\chi_{<t}"]\arrow[r,"\eta_{I_{<t},I_t}"',hook] & \Idc\FL(I_t,\kk)\arrow[u,"\chi_t",twoheadrightarrow]
\end{tikzcd}
\caption{Illustrating the proof of Theorem~\ref{T:Part3Ext}}
\label{Fig:Part3Ext}
\end{figure}
By Zorn's Lemma, there exists a maximal lower subset~$T'$ of~$T$, containing~$\set{\bot}$, on which the conclusion of Theorem~\ref{T:Part3Ext} holds.
Suppose, by way of contradiction, that $T'\neq T$ and let~$t$ be a minimal element of $T\setminus T'$;
so $T'\cup\set{t}$ is also a lower subset of~$T$.
Since the height of~$t$ is countable, so are the lattice $L_{<t}\eqdef\varinjlim_{s<t}L_s$ (with transition maps $\gf_{s,s'}$ where $s\leq s'<t$ and limiting maps $\gf_{s,<t}\colon L_s\to L_{<t}$ for $s<t$) and the set $I_{<t}\eqdef\bigcup\setm{I_s}{s<t}$.
The universal property of the colimit ensures the existence of unique $0$-lattice homomorphisms
 \[
 \eta_{I_{<t},I_t}\colon\Idc\FL(I_{<t},\kk)=
 \varinjlim_{s<t}\Idc\FL(I_s,\kk)\to\Idc\FL(I_t,\kk)
  \]
and $\gf_{<t,t}\colon L_{<t}\to L_t$, such that $\eta_{I_{<t},I_t}\circ\eta_{I_s,I_{<t}}=\eta_{I_s,I_t}$ and $\gf_{<t,t}\circ\gf_{s,<t}=\gf_{s,t}$ whenever $s<t$.
Further, the natural transformation $\vecm{\chi_s}{s<t}$ induces a unique $0$-lattice homomorphism
 \[
 \chi_{<t}\colon\Idc\FL(I_{<t},\kk)\to L_{<t}
 \]
such that $\chi_{<t}\circ\eta_{I_s,I_{<t}}=\gf_{s,<t}\circ\chi_s$ whenever $s<t$.
By Theorem~\ref{T:Part2Ext}, there exists a surjective lattice homomorphism $\chi_t\colon\Idc\FL(I_t,\kk)\twoheadrightarrow\nobreak L_t$ such that $\chi_t\circ\eta_{I_{<t},I_t}=\gf_{<t,t}\circ\chi_{<t}$.
Therefore, for each $s<t$,
 \[
 \chi_t\circ\eta_{I_s,I_t}=\chi_t\circ\eta_{I_{<t},I_t}\circ\eta_{I_s,I_{<t}}
 =\gf_{<t,t}\circ\chi_{<t}\circ\eta_{I_s,I_{<t}}=\gf_{<t,t}\circ\gf_{s,<t}\circ\chi_s
 =\gf_{s,t}\circ\chi_s\,.
 \]
This shows that our conclusion holds at~$T'\cup\set{t}$, in contradiction with the maximality assumption on~$T'$.
\end{proof}

This leads us to the following positive solution of the problem stated at the end of Wehrung~\cite{MVRS}.

\begin{theorem}\label{T:Al1hom}
Let~$\kk$ be a countable totally ordered division ring.
Then every completely normal distributive $0$-lattice~$L$ with at most~$\aleph_1$ elements is a surjective homomorphic image of~$\Idc{F}$ for some vector lattice~$F$ over~$\kk$.
\end{theorem}

\begin{proof}
Write~$L$ as the directed union of an ascending $\go_1$-sequence $\vec{L}=\vecm{L_{\xi}}{\xi<\go_1}$ of countable completely normal distributive $0$-lattices, with $L_0=\set{0}$.
Theorem~\ref{T:Part3Ext}, applied to the well-ordered chain~$\go_1$, yields an $\go_1$-indexed commutative diagram $\vec{F}=\vecm{F_{\xi},f_{\xi,\eta}}{\xi\leq\eta<\go_1}$ of $\kk$-vector lattices together with a natural transformation $\vec{\chi}\colon\Idc\vec{F}\todot\vec{L}$ all of whose components are surjective lattice homomorphisms.
Letting $F\eqdef\varinjlim\vec{F}$, the universal property of the colimit yields a surjective homomorphism from~$\Idc{F}$ onto~$L$.
\end{proof}

Due to Wehrung \cite[Corollary~9.5]{RAlg}, Theorem~\ref{T:Al1hom} cannot be generalized to \emph{uncountable} totally ordered division rings~$\kk$.
On the other hand, setting~$\kk$ as any countable Archimedean totally ordered field (for example the rationals), $\Idc{F}$ is identical to the \lidl\ lattice of the underlying \lgrp~of~$F$.
Hence,

\begin{corollary}\label{C:Al1hom}
Every completely normal distributive $0$-lattice~$L$ with at most~$\aleph_1$ elements is a surjective homomorphic image of~$\Idc{F}$ for some Abelian \lgrp~$F$.
\end{corollary}

By applying Stone duality for distributive $0$-lattices, we obtain the following formulation in terms of spectra.

\begin{corollary}\label{C:Al1homSp}
Every completely normal generalized spectral space with at most~$\aleph_1$ compact open sets embeds, as a spectral subspace, into the $\ell$-spectrum of an Abelian \lgrp.
\end{corollary}

Corollary~\ref{C:Al1hom} also strengthens Plo\v{s}\v{c}ica \cite[Theorem~3.2]{Plo21}, which states that every completely normal distributive $0$-lattice of cardinality at most~$\aleph_1$ is Cevian; that is, it carries a binary operation $(x,y)\mapsto x\sd y$ such that $x\leq y\vee(x\sd y)$,\linebreak $(x\sd y)\wedge(y\sd x)=0$, and $x\sd z\leq(x\sd y)\vee(y\sd z)$ for all $x$, $y$, $z$.
Indeed, $\Idc{G}$ is Cevian for any Abelian \lgrp~$G$, and any homomorphic image of a Cevian lattice is Cevian (cf. Wehrung \cite[\S~5]{Ceva}).

\begin{problem}
Let~$D$ be a completely normal distributive $0$-lattice such that for all $a,b\in D$ there exists a sequence $\vecm{c_n}{n<\go}$ from~$D$ such that for all $x\in D$, $a\leq b\vee x$ if{f} there exists $n<\go$ such that $c_n\leq x$ (in~\cite{MV1} we say that~$D$ has \emph{countably based differences}).
If $\card{D}=\aleph_1$, does $D\cong\Idc{G}$ for some Abelian \lgrp~$G$?
\end{problem}

The cases where $\card{D}\leq\aleph_0$ and $\card{D}\geq\aleph_2$ are settled in Wehrung \cite{MV1,Ceva}, in the positive and the negative, respectively (the counterexample constructed in~\cite{Ceva} is not even Cevian, thus it is not a homomorphic image of any~$\Idc{G}$). A Cevian counterexample (of size continuum plus) is constructed in Plo\v{s}\v{c}ica~\cite{Plo21}.


\providecommand{\noopsort}[1]{}\def\cprime{$'$}
  \def\polhk#1{\setbox0=\hbox{#1}{\ooalign{\hidewidth
  \lower1.5ex\hbox{`}\hidewidth\crcr\unhbox0}}} 
\providecommand{\bysame}{\leavevmode\hbox to3em{\hrulefill}\thinspace}
\providecommand{\MR}{\relax\ifhmode\unskip\space\fi MR }
\providecommand{\MRhref}[2]{%
  \href{http://www.ams.org/mathscinet-getitem?mr=#1}{#2}
}
\providecommand{\href}[2]{#2}

\end{document}